\documentclass[reqno]{amsart}
\usepackage{amssymb,amsmath,amsthm,bm,mathrsfs,paralist}
\usepackage{url,xcolor,units,graphicx,mathtools,bbold,url,amsrefs}
\usepackage[T1]{fontenc}
\RequirePackage[colorlinks,citecolor=blue,urlcolor=blue]{hyperref}
\usepackage{tikz,pgfplots}
\usepackage[cal=dutchcal,calscaled=1.05,
	scr=boondoxo,scrscaled=1.05,
	]{%
mathalfa}
\newcommand{\I}{\mathcal{I}}

\newcommand{\R}{\mathbb{R}}
\newcommand{\lip}{\text{\rm Lip}}

\newcommand{\dW}{\dot{W}}

\renewcommand{\d}{\mathrm{d}}
\renewcommand{\P}{\mathrm{P}}
\newcommand{\e}{\mathrm{e}}

\newcommand{\E}{\mathrm{E}}

\DeclareMathOperator{\Cov}{\mathrm{Cov}}
\newtheorem{proposition}{Proposition}
\newtheorem{theorem}[proposition]{Theorem}
\newtheorem{lemma}[proposition]{Lemma}

\newtheorem{assumption}[proposition]{Assumption}

\theoremstyle{definition}
\newtheorem{remark}[proposition]{Remark}
\numberwithin{equation}{section}
\numberwithin{proposition}{section}
\title[SPDEs with locally Lipschitz coefficients]{On the well-posedness of 
	SPDEs with locally Lipschitz coefficients}
	\thanks{Research supported by the Leverhulme Trust Fellowship IF-2025-040, the US-NSF grants 
	DMS-1855439 and DMS-2245242,
	the Spanish MINECO grant PID2022-138268NB-100, and
	Ayudas Fundacion BBVA a Proyectos de Investigaci\'on Cient\'ifica 2021}
\author[M. Foondun]{Mohammud Foondun}
\address{University of Strathclyde}
\email{mohammud.foondun@strath.ac.uk}
\author[D. Khoshnevisan]{Davar Khoshnevisan}
\address{The University of Utah}
\email{davar@math.utah.edu}
\author[E. Nualart]{Eulalia Nualart}
\address{Universitat Pompeu Fabra}
\email{eulalia.nualart@upf.edu}
\date{September 09, 2025}
\keywords{SPDEs, space-time white noise, existence and uniqueness}
\subjclass[2010]{60H15; 60H07, 60F05}
\begin{document}
\maketitle

\begin{abstract}
	\noindent We consider the stochastic partial differential equation,
	\begin{equation*}
		\partial_t u = \tfrac12 \partial^2_x u + b(u) + \sigma(u) \dot{W},
	\end{equation*}
	where $u=u(t\,,x)$ is defined for $(t\,,x)\in(0\,,\infty)\times\R$, and $\dot{W}$ 
	denotes space-time white noise. We prove that this SPDE is well posed solely
	under the assumptions that the initial condition $u(0)$ 
	is bounded and measurable, and $b$ and $\sigma$ are locally 
	Lipschitz continuous functions and have at most linear growth with
	regularly behaved local Lipschitz constants.
	Our method is based on a truncation 
	argument together with moment bounds and tail estimates of the truncated solution.
	The novelty of our method is in the pointwise nature of the truncation argument.
\end{abstract}

\section{Introduction}

We revisit the well-posedness of the solution
$u=\{u(t\,,x)\}_{t\ge0,x\in\R}$ of the following stochastic PDE (SPDE):
\begin{equation}\label{SHE}
	\partial_t u(t\,,x) = \tfrac12 \partial^2_x u(t\,,x) + b(t\,,u(t\,,x)) + \sigma(t\,,u(t\,,x)) \dot{W}(t\,,x),
\end{equation}
where $(t\,,x)\in(0\,,\infty)\times\R$, subject to $u(0\,,x) = u_0(x)$, and the forcing $\dot{W}$ is space-time white noise;
that is, $\dot{W}$ is a generalized Gaussian random field with mean zero and
\begin{equation*}
	\Cov [ \dW(t\,,x) \,, \dW(s\,,y) ] = \delta_0(t-s) \delta_0(x-y)
	\quad\text{for all } t,s\ge 0 \text{ and } x,y\in\R.
\end{equation*}
It is well known that \eqref{SHE} is well posed when $b$ and $\sigma$ 
are Lipschitz in their spatial variable uniformly in their
time variable; see for example
Dalang \cite{Dalang} and Walsh \cite{Walsh}. We aim to extend this 
result to the setting where ``Lipschitz'' is
replaced by ``local Lipschitz with linear growth.''

The present undertaking yields an infinite-dimensional version of one of the 
foundational results of the theory of stochastic differential equations 
(SDEs) which asserts that
if $Y=\{Y_t\}_{t\ge0}$ denotes a standard Brownian
motion on $\R$ and $x_0\in\R$ is non random, then the one-dimensional 
It\^o-type SDE
$$\d X_t = b(t\,,X_t)\,\d t + \sigma(t\,,X_t)\,\d Y_t \qquad (t>0)$$
subject to $X_0=x_0$ has a unique solution provided only that 
$b(t\,,x)$ and $\sigma(t\,,x)$
are locally Lipschitz in $x$ with at-most linear growth, all valid
uniformly in $t$. The preceding follows from 
the more classical result about SDEs with Lipschitz coefficients and a stopping
time argument. See for example Exercise (2.10) of Revuz
and Yor \cite[p.\ 383]{RevuzYor}. 
It is possible to use essentially the same method in order to extend the preceding to the multidimensional
setting where $Y$ denotes a $d$-dimensional Brownian motion, $b:(0\,,\infty)\times\R^d
\to\R^d$ and $\sigma:(0\,,\infty)\times\R^d\to(\R^d)^2$. 

An infinite-dimensional extension
appears in a forthcoming monograph by Robert C. Dalang and Marta
Sanz-Sol\'e \cite{bookDS}. There,
Dalang and Sanz-Sol\'e show the well-posedness of
\eqref{SHE}, where instead $(t\,,x)\in (0\,,\infty)\times I$
in the case that $I\subset\R$ is a bounded interval (together with any of the usual boundary conditions).
A key step in their analysis is the observation that if $b$ and $\sigma$ are locally Lipschitz in their
space variable, uniformly in the time variable, then \eqref{mild_SHE} below has a predictable
random-field solution up to the stopping time
$$\tau(I) = \inf \{ t >0: \sup_{x \in  I} \vert  u(t\,,x) \vert =\infty\} \qquad[\inf\varnothing=\infty],$$
and $\P\{\tau(I) >0\}=1$ when $I\subset\R$ is bounded.
Such stopping-time arguments  are typically not used to study equations on 
$\R_+\times\R$ because one expects the resulting solution, 
if there is one, to be unbounded with probability one.
Therefore, in order to produce solutions to the SPDE \eqref{SHE} one needs to introduce  a different  argument.

The purpose of this article is to provide one such argument:
We use truncation, as has been done previously, but
replace the stopping-time argument by pointwise tail probability estimates for the truncated solution.
In this way we are able to show that, for a wide family of locally Lipschitz functions $b$ and $\sigma$
of linear growth, and with high probability, the truncated solution is
not  large, uniformly in the truncation level.  

Our method is elementary as it uses only standard SPDE estimates, 
and more significantly works without need for
any unnecessary technical assumptions. But we are quick to mention that 
other potential approaches to such problems already exist
in the literature: It
might be for example possible to adapt methods of finite-dimensional SDEs, such as 
weak existence (in the  probabilistic  sense) followed by strong uniqueness via a 
Yamada-Watanabe (see Kurtz  \cite{K14}) or Gy\"ongy-Krylov (see Gy\" ongy and  
Krylov \cite{GK96}) type arguments. In the infinite-dimensional context of 
SPDEs, weak existence for  \eqref{SHE}
was established in Shiga \cite[Theorem 2.6]{S94}  when additionally
$b(0) \geq 0$ and $\sigma(0)=0$. And  Mytnik \emph{et al} 
\cite[Theorem 1.2]{MPS06} establish the weak existence for \eqref{SHE}
when  $b \equiv 0$, $\sigma=$ H\"older continuous with linear growth, all
forced by high-dimensional, spatially correlated noise.
There is also a  method that
involves transforming the SPDE into a random integral equation which can 
in turn be compared with a deterministic PDE;
see Salins \cite{Salins}. This method does not lend itself easily to the 
case that $\sigma$ is non constant unless there are additional
contraints such as $b(0)=0$, $\sigma(0)=0$, and more stringent conditions on the initial function;
see Chen \emph{et al} \cite{CFHS} for the current state of the art of that method.

We impose the following assumptions on the initial profile $u_0$
and the coefficients $b$ and $\sigma$ in \eqref{SHE}:
\begin{assumption}\label{cond-initial}
	$u_0:\R\to\R$ is non-random, bounded, and measurable.
\end{assumption}

\begin{assumption}\label{cond-dif}
	The functions $b:(0\,,\infty)\times\R \rightarrow \R$ and 
	$\sigma:(0\,,\infty)\times\R \rightarrow \R$ are locally 
	Lipschitz continuous in their space variable with at  most linear growth,
	uniformly in their time variable.
	In other words, $0<\lip_n(b), \lip_n(\sigma) <\infty$
	and $0 \leq  L_b, L_{\sigma} <\infty,$
	for all real numbers $n>0$ where, for every space-time function $\psi$,
	\begin{equation}\label{L:lip}\begin{split}
		L_\psi &=  \sup_{t>0}\sup_{x\in\R} \frac{|\psi(t\,,x)|}{1+|x|},\
			\lip_n(\psi) =  \sup_{t>0}\sup_{\substack{x,y\in[-n,n]\\
			x\neq y}}  \frac{|\psi(t\,,x)-\psi(t\,,y)|}{|y-x|}.
	\end{split}\end{equation}
\end{assumption}

Let us recall that a random field solution to \eqref{SHE} is a 
predictable random field  $u=\{u(t\,,x)\}_{t \geq 0, x \in \R}$ that satisfies the following integral equation:
\begin{equation}\label{mild_SHE}
	u(t\,,x) = (p_t*u_0)(x) + \I_b(t\,,x) + \I_\sigma(t\,,x),
\end{equation}
where the symbol ``$*$'' denotes convolution,
\begin{equation*}
	p_r(z) = (2\pi r)^{-1/2}\exp\{-z^2/(2r)\}\qquad\text{for all $r>0$ and $z\in\R$},
\end{equation*}
and $\I_b$ and $\I_\sigma$ are the following random fields,
\begin{equation}\label{I_b}\begin{split}
	\I_b(t\,,x) &= \int_{(0,t)\times\R} p_{t-s}(y-x) b(s\,,u(s\,,y))\,\d s\,\d y,\\	
	\I_\sigma(t\,,x) &= \int_{(0,t)\times\R} p_{t-s}(y-x) \sigma(s\,,u(s\,,y))\,W(\d s\,\d y),
\end{split}\end{equation}
where the second (stochastic) integral is understood in the sense of Walsh \cite{Walsh}.

As was mentioned in the Introduction, we will use a truncation argument.
For every real number $N>0$, we define $b_N:(0\,,\infty)\times\R\to\R$ 
and $\sigma_N:(0\,,\infty)\times\R\to\R$ as follows: For all $t>0$
and $\psi:\R_+\times\R\to\R$ let
\[
	\psi_N(t\,,x) = 
	\psi(t\,,x) {\bf 1}_{\{-\e^N\le x\le \e^N\}}
	+\psi(t\,,\e^N){\bf 1}_{\{ x> \e^N\}}+
	+\psi(t\,,-\e^N){\bf 1}_{\{x<-\e^N\}}.
\]
We will need the following assumption on the Lipschitz coefficients of 
$b_N$ and $\sigma_N$. Set 
\begin{equation}\label{L_N}
	L_{N,b} = \lip_{\exp(N)}(b) \quad \text{ and }  \quad L_{N,\sigma} = \lip_{\exp(N)}(\sigma),
\end{equation}

\begin{assumption}\label{cond:lip}
	If $L_\sigma>0$ then we assume that	
	\begin{equation}\label{cond:LL:1}
		L_{N,\sigma} = \mathscr{o}(N^{3/8})
		\quad\text{and}\quad
		L_{N,b} / L_{N,\sigma}^4 = \mathcal{O}(1)\qquad
		\text{as $N\to\infty$}.
	\end{equation}
	If $\sigma$ is bounded, then we assume that
	\begin{equation}\label{cond:LL:2}
		L_{N,\sigma} = \mathscr{o}(\e^{N/2})
		\quad\text{and}\quad
		L_{N,b} / L_{N,\sigma}^4 = \mathcal{O}(1)\qquad
		\text{as $N\to\infty$}.
	\end{equation}
\end{assumption}

It has been widely believed for a long time that Theorem \ref{th:exists} 
ought to hold solely under the assumptions of locally Lipschitz coefficients
$b$ and $\sigma$ with linear growth. The following is the main result
of this paper, and
makes some partial success toward the resolution of this long-standing problem.

\begin{theorem}\label{th:exists}
	If $(b\,,\sigma)$ satisfy Assumption \ref{cond:lip}, then
	\eqref{SHE} has a unique random-field solution that satisfies the following:
	$$
		\sup_{t \in [0,T]}\sup_{x\in\R}\E\left( |u(t\,,x)|^k\right) < \infty
		\quad\text{for all $T>0$ and $k\ge1$}.
	$$
\end{theorem}

Next we make some remarks in order to highlight the limitations and scope
of Theorem \ref{th:exists} and its proof.

\begin{remark}
	The condition $L_{N,b}/L_{N,\sigma}^4=\mathcal{O}(1)$ has new content
	only when $L_{N,\sigma}\to\infty$ as $N\to\infty$. Indeed, if $\sup_N L_{N,\sigma}<\infty$,
	then the first stated condition ensures that 
	both $b$ and $\sigma$ are globally Lipschitz continuous.
\end{remark}

\begin{remark}
	Condition \eqref{cond:LL:1} is quite restrictive as it necessarily implies that
	$\lip_n(\sigma)=\mathscr{o}(|\log n|^{3/8})$ 
	and $\lip_n(b) = \mathscr{o}(|\log n|^{3/2})$ as $n\to\infty$;
	see \eqref{L_N}.
\end{remark}

\begin{remark}[Oscillatory diffusion coefficients]
	Theorem \ref{th:exists} has stronger content when $\sigma$ is bounded for
	then \eqref{cond:LL:2} allows for a wide class of drifts $b$ when $\sigma$ is
	highly oscillatory. First, note that the first condition in \eqref{cond:LL:2}
	is equivalent to the statement that $\lip_n(\sigma)=\mathscr{o}(\sqrt n)$
	as $n\to\infty$. Now consider an oscillatory diffusion coefficient
	such as
	$\sigma(x)= \sin(1000(1+|x|)^{1/4})$ for all $x\in\R$;
	see Figure \ref{Yahoo}. Then, $\sigma$ satisfies \eqref{cond:LL:2} together with
	every drift function that satisfies $\lip_n(b) = \mathscr{o}(n)$.
	\begin{figure}[h!]\begin{tikzpicture}
		\begin{axis}[
		    axis lines = left,
		    xlabel = \(x\),
		    ylabel = {\(\sigma(x)= \sin(1000(1+x)^(1/4))\)},
		]
		\addplot [
		    domain=0:50, 
		    samples=550, 
		    color=blue,
		    line width=0.2mm,
		]
		{sin( 1000*( 1 + x)^(1/4))};
		\end{axis}
		\end{tikzpicture}\caption{An example of how Assumption \ref{cond:lip}
		is less restrictive for the drift when $\sigma$ fluctuates wildly}\label{Yahoo}\end{figure}
\end{remark}

\begin{remark}
	The method of proof of Theorem \ref{th:exists} allows for minor improvements of the first parts of 
	\eqref{cond:LL:1} and \eqref{cond:LL:2}. For example, the first condition
	in \eqref{cond:LL:1} can be improved slightly to
	$L_{N,\sigma}=\mathscr{o}(N^{2/3})$ by, instead of choosing the parameter
	$k$ as in \eqref{beta}, choosing it as $k(N)=A_1 (N/T)^{1/2} L_{N,\sigma}^{-4/3}$
	for a suitably small constant $A_1$ that is independent of $N$ and $T$.
	Then one obtains the main part of the proof -- that is \eqref{goal:1} below --
	from the fact that
	$\|u_{N+1}(t\,,x)-u_N(t\,,x)\|_k \le \|u_{N+1}(t\,,x)-u_N(t\,,x)\|_{k(N)}$ for all large $N$
	[Jensen's inequality]. We have omitted the details of such improvements as they
	appear to be primarily technical in nature.
\end{remark}

Let us end the Introduction with a brief outline of the paper. 
In \S2 we introduce the truncated solution and develop some moment  and tail estimates. 
The remaining details of the proof of Theorem \ref{th:exists} are
gathered  in \S3, and use the earlier results of the paper.

Throughout this paper, we write 
$\|X\|_p = \{ \E(|X|^p)\}^{1/p}$
for all $p\ge1$ and $X\in L^p(\Omega)$.
For every space-time function $f:(0\,,\infty)\times\R\to\R$, 
$\lip(f)$ denotes the optimal Lipschitz constant of $f$; that is,
\[
	\lip(f) =  \sup_{t>0}\sup_{a,b\in\R:a\neq b} |f(t\,,b)-f(t\,,a)|/|b-a|.
\]
In particular, $f$ is globally Lipschitz continuous in $x$, uniformly in $t$, iff $\lip(f)<\infty$.
If $f$ depends only on a spatial variable $x$, then $\lip(f)$ still makes sense provided that we extend $f$
to a space-time function as follows $f(t\,,x) = f(x)$, in the usual way.

\section{Truncation, and preliminary results}
In this section, we give more information about the truncation 
argument referred to verbally in the introduction. We begin by 
recalling the definition of $b_N$ and $\sigma_N$ and note 
that they are globally Lipschitz functions. In fact,
\[
	\adjustlimits\sup_{t>0}\sup_{x\in\R}\frac{|b_N(t\,,x)|}{1+|x|} \le L_{b}<\infty \quad \text{ and }
	\quad   \adjustlimits\sup_{t>0}\sup_{x\in\R}\frac{|\sigma_N(t\,,x)|}{1+|x|} \le L_{\sigma}<\infty,
\]
uniformly in $N>0$.
Thanks to standard theory
\cite{Dalang,Walsh}, the following SPDE has a  predictable mild solution: For $(t\,,x)\in(0\,,\infty)\times\R$,
\begin{equation}
	\partial_t u_N(t\,,x) = \tfrac12 \partial^2_x u_N(t\,,x) 
	+ b_N(t\,,u_N(t\,,x)) + \sigma_N(t\,,u_N(t\,,x)) \dot{W}(t\,,x),
	\label{SHE:N}
\end{equation}
subject to $u_N(0\,,x) = u_0(x)$. Moreover this solution is unique subject to
\[
	\sup_{t\in(0,T)}\sup_{x\in\R}\E\left( |u_N(t\,,x)|^k\right)<\infty
	\quad\text{for all $N,T>0$ and $k\ge1$}
\]
and jointly continuous in $L^k(\Omega)$.  See Dalang and Sanz-Sol\'e 
\cite[Theorems 4.2.1 and 4.2.8]{bookDS}.
As usual, \eqref{SHE:N} is short-hand for the random integral equation,
\begin{equation}\label{mild}
	u_N(t\,,x) = (p_t*u_0)(x) + \I^N_{b_N}(t\,,x)
	+ \I^N_{\sigma_N}(t\,,x),
\end{equation}
where $\I^N_{b_N}$ and $\I^N_{\sigma_N}$ are the truncated
analogous of the integrals in \textcolor{black}{\textcolor{black}{\eqref{I_b}} and are given by
\begin{equation}\label{I_b^N}\begin{split}
	\I^N_{b_N} (t\,,x) &= \int_{(0,t)\times\R} p_{t-s}(y-x)b_N(s\,,u_N(s\,,y))
		\,\d s\,\d y,\\
	\I_{\sigma_N}^N(t\,,x) &= \int_{(0,t)\times\R} p_{t-s}(y-x)\sigma_N(s\,,u_N(s\,,y))
		\,W(\d s\,\d y).
\end{split}\end{equation}}

The next result 
yields a moment estimate of the truncated solution which is similar to that in 
Theorem 6.3.2 of \cite{bookDS} for more general SPDEs.  We describe 
the details in order to provide the explicit constants, and parameter dependencies,
of the bound.
\begin{proposition}\label{pr:moments}
	If $L_{\sigma}>0$, then
	\[
		\sup_{N>0}\sup_{x\in\R}\E\left( |u_N(t\,,x)|^k\right) \le
		4^k(\|u_0\|_{L^\infty(\R)}+ 1)^k\e^{128 L_{\sigma}^4 k^3t}.
	\]
	uniformly for all $t>0$ and $k\ge \max(2\,, L_b^{1/2} L_{\sigma}^{-2})$.
\end{proposition}

\begin{proof}
	Choose and fix $N,t>0$ and $x\in\R$. Owing to \eqref{mild}, we may write
	\begin{equation}\label{u_N=I_1+I_2}
		\| u_N(t\,,x)\|_k \le \|u_0\|_{L^\infty(\R)} + I_1 + I_2,
	\end{equation}
	where $I_1 = \| \I^N_{b_N}(t\,,x) \|_k$ and 
	$I_2 = \| \I^N_{\sigma_N}(t\,,x)\|_k.$
	We begin by estimating $I_1$ as follows.
	Thanks to \eqref{L:lip} and the Minkowski inequality
	and the fact that $$\|b_N(s\,,u_N(s\,,y))\|_k\le L_b
	(1+ \| u_N(s\,,y) \|_k ),$$
	we can see that
	\begin{align*}
		I_1 &\le\int_0^t\d s\int_{-\infty}^\infty\d y\
			p_{t-s}(y-x) \left\| b_N(s\,,u_N(s\,,y)) \right\|_k\\
		&\le L_b \left[t +  \int_0^t \sup_{y\in\R} \| u_N(s\,,y)\|_k\,\d s\right].
	\end{align*}
	For every space-time random field $Z=\{Z(t\,,x);\, t\ge0,x\in\R\}$ and for 
	all real numbers $k\ge1$ and $\beta>0$ define
	\begin{equation}\label{N}
		\mathcal{N}_{k,\beta}(Z) = \sup_{t\ge0}\sup_{x\in\R} \e^{-\beta t}\| Z(t\,,x)\|_k.
	\end{equation}
	It follows that
	\[
		I_1 \le L_b \left[t + \mathcal{N}_{k,\beta}(u_N)\int_0^t\e^{\beta s}\,\d s\right]
		\le L_b \left[t + \beta^{-1}\e^{\beta t} \mathcal{N}_{k,\beta}(u_N)\right].
	\]
	Since $t\exp(-\beta t)\le  (\e\beta)^{-1}<\beta^{-1}$, this leads to the following
	bound for $I_1$:
	\begin{equation}\label{I1}
		I_1 \le L_b \beta^{-1} \e^{\beta t}
		\left[ 1 + \mathcal{N}_{k,\beta}(u_N)\right].
	\end{equation}
	We bound $I_2$ using the asymptotically optimal form of the Burkholder-Davis-Gundy 
	inequality (see \cite{minicourse}) as follows:
	\begin{equation}\label{I2^2}
		I_2^2 \le 4k\int_0^t\d s\int_{-\infty}^\infty\d y\
		\left[p_{t-s}(y-x)\right]^2 \left\| \sigma_N(s\,,u_N(s\,,y))\right\|_k^2.
	\end{equation}
	Thanks to \eqref{L:lip} and the fact that 
	$\| \sigma_N(s\,,u_N(s\,,y))\|_k^2\le 2L_\sigma^2(1+\|u_N(s\,,y)\|_k^2)$,
	\[
		I_2^2  \le 8kL_{\sigma}^2\int_0^t\d s\int_{-\infty}^\infty\d y\
			\left[p_{t-s}(y-x)\right]^2 \left(1 + \| u_N(s\,,y)\|_k^2\right).
	\]
	Basic properties of the heat kernel imply that
	\begin{equation}\label{L2:p}
		\|p_r\|_{L^2(\R)}^2 = (p_r*p_r)(0) = p_{2r}(0)=\tfrac12(\pi r)^{-1/2}
		\qquad\text{for every $r>0$}.
	\end{equation}
	Therefore,
	\begin{align*}
		I_2^2 &\le \frac{4k L_{\sigma}^2}{\sqrt\pi}\int_0^t\frac{\d s}{\sqrt{t-s}} + 
			\frac{4k L_{\sigma}^2}{\sqrt{\pi}}\int_0^t\sup_{y\in\R}\|u_N(s\,,y)\|_k^2\,
			\frac{\d s}{\sqrt{t-s}}\\
		&\le \frac{4k L_{\sigma}^2}{\sqrt\pi}\int_0^t\frac{\d s}{\sqrt{s}} + 
			\frac{4k L_{\sigma}^2\left[ \mathcal{N}_{k,\beta}(u_N)\right]^2\e^{2\beta t}}{\sqrt{\pi}}\int_0^t
			\frac{\e^{-2\beta (t-s)}}{\sqrt{t-s}}\,\d s.
	\end{align*}
	Since $\int_0^t s^{-1/2}\,\d s\le 
	\exp(2\beta t)\int_0^\infty s^{-1/2}\exp(-2\beta s)\,\d s = \sqrt{\pi/(2\beta)}\exp(2\beta t),$
	we are led to the following:
	\[
		I_2^2 \le 4k\e^{2\beta t}L_{\sigma}^2
		\left( 1 +  \mathcal{N}_{k,\beta}(u_N)^2\right)/\sqrt{2\beta}.
	\]
	We prefer to simplify the preceding slightly more, using the inequality $\sqrt{l^2+n^2}\le |l|+|n|$
	-- valid for all $l,n\in\R$ -- as follows:
	\begin{equation}\label{I2}
		I_2 \le 2\sqrt{k}(2\beta)^{-1/4}\e^{\beta t}L_{\sigma}
		\left( 1 +  \mathcal{N}_{k,\beta}(u_N) \right).
	\end{equation}
	We can combine \eqref{I1} and \eqref{I2} to find that
	\begin{align*}
		&\| u_N(t\,,x)\|_k\\
		&\le \|u_0\|_{L^\infty(\R)}
			+ L_b\beta^{-1} \e^{\beta t} \left[ 1 + \mathcal{N}_{k,\beta}(u_N)\right]
			+ 2\sqrt{k}(2\beta)^{-1/4} \e^{\beta t} L_{\sigma}
			\left( 1 +  \mathcal{N}_{k,\beta}(u_N) \right)\\
		&\le \|u_0\|_{L^\infty(\R)}
			+ \e^{\beta t}\left( L_b \beta^{-1} +  2\sqrt{k} L_{\sigma}(2\beta)^{-1/4}\right)
			\left(1 +  \mathcal{N}_{k,\beta}(u_N)\right).
	\end{align*}
	This implies that
	\[
		\mathcal{N}_{k,\beta}(u_N) \le \|u_0\|_{L^\infty(\R)}
		+\left(L_b\beta^{-1} + 2\sqrt{k}L_{\sigma}(2\beta)^{-1/4}\right)  
		\left(1 +  \mathcal{N}_{k,\beta}(u_N)\right).
	\]
	Set $\beta = 128 k^2L_{\sigma}^4$ to see that
	$\beta^{-1} L_b + (2\beta)^{-1/4}2\sqrt{k}L_\sigma\le \frac34$
	for this particular choice of $\beta$.
	Solve for $\mathcal{N}_{k,128 k^2 L_{\sigma}^4}(u_N)$ in order to find that
	$\mathcal{N}_{k,128 k^2L_{\sigma}^4}(u_N)\le4(\|u_0\|_{L^\infty(\R)} + 1).$
	This is another way to state the proposition.
\end{proof}

\begin{remark}
	The requirement that $k\geq \max(2, L_b^{1/2}L_{\sigma}^{-2})$ is 
	a technical consequence of the proof. 
	We can always choose, without incurring loss of generality,   
	$L_\sigma$ large enough to ensure that $L_b^{1/2}L_{\sigma}^{-2}\le2$,
	in order to ensure that the result stated in the proposition holds for all $k\geq 2$.
\end{remark}

\textcolor{black}{The conclusion of Proposition \ref{pr:moments} can be 
improved when $\sigma$ is constant. The following
takes care of that case, and at the same time improves Proposition 
\ref{pr:moments} when $\sigma$
is bounded but $\text{Lip}(\sigma)>0$.}

\begin{proposition}\label{pr:moments:bdd}
	If $\sigma:(0\,,\infty)\times\R\to\R$ is  bounded, then
	\[
		\adjustlimits
		\sup_{N>0}\sup_{x\in\R}\E\left( |u_N(t\,,x)|^k\right) \le
		4^k\e^{2L_bkt}\left( \|u_0\|_{L^\infty(\R)} + 
		\|\sigma\|_{L^\infty(\R_+\times\R)}t^{1/4}
		+ 1\right)^k k^{k/2},
	\]
	uniformly for all $t>0$ and $k\ge2$.
\end{proposition}

\begin{proof}
	We modify the proof of Proposition \ref{pr:moments} by first observing that
	\eqref{u_N=I_1+I_2} remains valid, and so does \eqref{I1}. We start with \eqref{I2^2}
	and estimate $I_2$, using \eqref{L2:p}, simply as follows:
	\begin{align*}
		I_2^2 &
			\le 4k\|\sigma\|_{L^\infty(\R_+\times\R)}^2\int_0^t\|p_r\|_{L^2(\R)}^2\,\d r =
			 2k\pi^{-1/2}\|\sigma\|_{L^\infty(\R_+\times\R)}^2\int_0^t \d r/\sqrt r\\
		&=4k\|\sigma\|_{L^\infty(\R_+\times\R)}^2\sqrt{t/\pi} \leq  
			4k\|\sigma\|_{L^\infty(\R_+\times\R)}^2\sqrt t.
	\end{align*}
	This yields
	\[
		\|u_N(t\,,x)\|_k \le \|u_0\|_{L^\infty(\R)} + 
		 L_b\beta^{-1}\exp(\beta t)\left[ 1+\mathcal{N}_{k,\beta}(u_N)\right]
		+ 2\sqrt{k}\|\sigma\|_{L^\infty(\R_+\times\R)}t^{1/4}.
	\]
	Since the right-hand side does not depend on $(t\,,x)$, we divide by $\exp(\beta t)$
	and optimize over $(t\,,x)$ to find that
	\[
		\mathcal{N}_{k,\beta}(u_N)\le \|u_0\|_{L^\infty(\R)} + 
		2\sqrt{k}\|\sigma\|_{L^\infty(\R_+\times\R)}t^{1/4}
		+ L_b\beta^{-1}\left[ 1+\mathcal{N}_{k,\beta}(u_N)\right],
	\]
	uniformly for all real numbers $k\ge2$, $N,\beta>0$. Set $\beta=2L_b$ 
	and solve for $\mathcal{N}_{k,2L_b}(u_N)$ to find that
	\begin{align*}
		\mathcal{N}_{k,2L_b}(u_N) &\le 2\|u_0\|_{L^\infty(\R)} 
			+ 4\sqrt{k}\|\sigma\|_{L^\infty(\R_+\times\R)}t^{1/4}
			+ 1 \\
		& \le 4\left( \|u_0\|_{L^\infty(\R)} +  \|\sigma\|_{L^\infty(\R_+\times\R)}
			t^{1/4} + 1\right) \sqrt k,
	\end{align*}
	which is another way to state the result.
\end{proof}

The next results shows some tail estimates of the  truncated solution.
\begin{proposition}\label{pr:prob:tail}
	If $L_{\sigma}>0$, then
	\begin{equation}\label{eq:tail:1}
		\P\left\{ |u_{N+1}(t\,,x)| \ge \e^N\right\}
		\le \exp\left( - \frac{N^{3/2}}{64 L_{\sigma}^2 \sqrt{t}}\right),
	\end{equation}
	uniformly for all $t>0$,  $x\in\R$, and
	\[
		N\ge  4 \log ( 4(\|u_0\|_{L^\infty(\R)} +1))\vee
		256 t\max\left(4L_{\sigma}^4, L_b\right).
	\]
	If $\sigma\in L^\infty(\R_+\times\R)$, then
	\begin{equation}\label{eq:tail:2}
		\P\left\{ |u_{N+1}(t\,,x)| \ge \e^N \right\}
		\le \exp\left( - \frac{\e^{2N-4L_bt}}{32\e\left(\|u_0\|_{L^\infty(\R)} 
		+ \|\sigma\|_{L^\infty(\R_+\times\R)}t^{1/4}+1\right)^2}\right),
	\end{equation}
	uniformly for all $t>0$,  $x\in\R$, and
	\[
		N\ge \tfrac12\log32 + 2L_bt+ \tfrac12 + \log\left( \|u_0\|_{L^\infty(\R)} + 
		\|\sigma\|_{L^\infty(\R_+\times\R)}t^{1/4}+1\right).
	\]
\end{proposition}

\begin{proof}
	Let us first consider the case that $L_{\sigma}>0$.
	Proposition \ref{pr:moments} and Chebyshev's inequality together ensure that
	\begin{equation*} \begin{split}
		\P\left\{ |u_{N+1}(t\,,x)| \ge \e^N\right\} 
		&\le\e^{-kN}\E\left( |u_{N+1}(t\,,x)|^k\right) \\
		&\le \left( 4(\|u_0\|_{L^\infty(\R)} +1)\right)^k\e^{-kN+128L_{\sigma}^4 k^3t},
		\end{split}
	\end{equation*}
	uniformly for all real numbers $N,t>0$, $x\in\R$, and 
	all real  $k\ge \max(2\,, L_b^{1/2}L_{\sigma}^{-2})$.
	Set $k=\sqrt{A N}$ -- where the value of $A>0$ is to be determined --
	and $C=4(\|u_0\|_{L^\infty(\R)} +1)$ in order to see that
	\[
		\P\left\{ |u_{N+1}(t\,,x)| \ge \e^N\right\} 
		\le  \exp\left( -\left\{ \sqrt{A} - 128 L_{\sigma}^4 A^{3/2}t
		- \frac{\sqrt{A}\log C}{N}\right\}N^{3/2}\right).
	\]
	We use this with $A = (256L_\sigma^4 t)^{-1}$,
	and observe that  $k\ge\max(2\,,\sqrt{L_b}L_{\sigma}^{-2})$ iff
	$N\ge 256 t\max(4L_{\sigma}^4, L_b)$.
	It follows that
	\begin{align*}
		\P\left\{ |u_{N+1}(t\,,x)| \ge \e^N\right\}
			&\le\exp\left( -\left\{\frac{ 128 L_{\sigma}^4 t}{[256L_{\sigma}^4t]^{3/2}}
			- \frac{\log C}{N\sqrt{256 L_{\sigma}^4t}}\right\}N^{3/2}\right)\\
		&= \exp\left( -\left\{\frac12
			- \frac{\log C}{N}\right\}
			\frac{N^{3/2}}{(256 t)^{1/2} L_{\sigma}^2}\right),
	\end{align*}
	which has the desired outcome, when $L_{\sigma}>0$, provided additionally that $N\ge 4\log C$.
	The case $\sigma\in L^\infty(\R)$ is proved similarly but rests on Proposition \ref{pr:moments:bdd}
	instead of \ref{pr:moments}, viz.,
	\[
		\P\left\{ |u_{N+1}(t\,,x)| \ge \e^N\right\} 
		\le\e^{-kN}\E\left( |u_{N+1}(t\,,x)|^k\right) 
		\le C^k\e^{-kN} k^{k/2},
	\]
	valid uniformly for all $N,t>0$, $x\in\R$, $k\ge2$,
	where
	\[
		C =C(t\,,u_0\,,L_b) = 4\e^{2L_bt}\left( \|u_0\|_{L^\infty(\R)} + \|\sigma\|_{L^\infty(\R_+\times\R)}t^{1/4}+1\right).
	\]
	We apply the preceding with the particular choice $k= C^{-2} \exp\{2N-1\}$ and
	compute a bit in order to finish.
\end{proof}

Finally, the following real-variable lemma will be helpful to us.

\begin{lemma}\label{lem:sum}
	Consider a function $f:(0\,,\infty)\to(0\,,\infty)$ and
	an increasing function $g:(0\,,\infty)\to(0\,,\infty)$. If
	there exists $\beta,T_0>0$ such that
	\[
		\sup_{t\in(0,T]}\left[ \e^{-\beta t}f(t)\right] \le \e^{-\beta T} g(T)
		\qquad\forall T\in(0\,,T_0),
	\]
	then $\sup_{(0,T]}f\le g(T)$ for every $T\in(0\,,T_0)$.
\end{lemma}

\begin{proof}
	Since $\e^{-\beta T}f(T)\le \e^{-\beta T}
	g(T)$, we cancel the exponentials to deduce the result from the
	monotonicity of $g$.
\end{proof}

\section{Proof of Theorem \ref{th:exists}}

\begin{proof}[Proof of existence]
	We first prove the theorem in the case that $L_\sigma>0$. Until we
	mention to the contrary, we therefore assume tacitly that $L_\sigma>0$.
	Thus, we also assume that \eqref{cond:LL:1} holds.
	
	Let us choose and fix an arbitrary number $T>0$.
	Our primary goal is to prove that
	\begin{equation}\label{goal:1}
		\sum_{N=1}^\infty \adjustlimits\sup_{t\in[0,T]}\sup_{x\in\R}
		\|u_{N+1}(t\,,x)-u_N(t\,,x)\|_k <\infty\qquad\forall k\ge1.
	\end{equation}
	Because $T>0$ can be as large as needed,
	\eqref{goal:1} implies that the random variable
	$u(t\,,x)=\lim_{N\to\infty} u_N(t\,,x)$ exists in $L^2(\Omega)$, say,
	pointwise in $(t\,,x)$.
	From there it is not difficult to prove that the thus-defined
	random field $u$ is a mild solution to \eqref{SHE}.
	
	With the preceding paragraph in mind,
	we now adjust \eqref{N} slightly and define the following local-in-time norms:
	\begin{equation}\label{N'}
		\mathcal{N}_{k,\beta,T}(Z) = \adjustlimits
		\sup_{t\in(0,T]}\sup_{x\in\R} \e^{-\beta t}\| Z(t\,,x)\|_k,
	\end{equation}
	defined for every $\beta, T>0$, $k\ge 1$, and all space-time random fields $Z$.

	Recall \eqref{L_N} and note that $N\mapsto \lip_N(b)$ and $N\mapsto \lip_N(\sigma)$
	are nondecreasing and $b$ and $\sigma$ are globally Lipschitz respectively when
	$\lim_{N\to\infty} \lip_N(b)<\infty$ and $\lim_{N\to\infty}\lip_N(\sigma)<\infty$. 
	Therefore, Condition \eqref{cond:LL:1} ensures that we can assume without any loss in generality
	that 
	\begin{equation}\label{L:N:sigma}
		\lim_{N\to\infty} L_{N,\sigma} =\infty,
	\end{equation}
	for $b$ and $\sigma$ will be globally Lipschitz otherwise, in which case there is nothing to prove.
	From now on, condition \eqref{L:N:sigma} is assumed to hold.
	
	Thanks to \eqref{mild} we can write, for every $k\ge1$ [not necessarily an integer],
	$t>0$, and $x\in\R$,
	\begin{equation}\label{u-u=I1I2}
		\left\| u_{N+1}(t\,,x) - u_N(t\,,x)  \right\|_k \le I_1 + I_2,
	\end{equation}
	where
	\begin{align*}
		I_1 & =\int_{(0,t)\times\R} p_{t-s}(y-x)\left\| b_{N+1}(s\,,u_{N+1}(s\,,y))
			- b_N(s\,,u_N(s\,,y))\right\|_k\d s\,\d y,\\
		I_2 & =\left\|  \int_{(0,t)\times\R} p_{t-s}(y-x)
			\left[\sigma_{N+1}(s\,,u_{N+1}(s\,,y))
			-\sigma_N(s\,,u_N(s\,,y))\right]W(\d s\,\d y)\right\|_k.
	\end{align*}
	Recall \eqref{N} and for every $N,s>0$ and $y\in\R$ consider the event
	\begin{equation*}
		G_{N+1}(s\,,y) = \left\{ \omega\in\Omega:\, |u_{N+1}(s\,,y)|(\omega) \le \e^N\right\}.
	\end{equation*}
	On one hand, 
	\begin{align*}
		&\left\| \left[ b_{N+1}(s\,,u_{N+1}(s\,,y))
			- b_N(s\,,u_N(s\,,y))\right] \mathbf{1}_{G_{N+1}(s\,,y)}\right\|_k\\
		&\hskip1in \le\left\| b_N(s\,,u_{N+1}(s\,,y))
			- b_N(s\,,u_N(s\,,y))\right\|_k\\
		&\hskip1in\le L_{N,b}\left\| u_{N+1}(s\,,y) - u_N(s\,,y) \right\|_k
			\le L_{N,b}\e^{\beta s} \mathcal{N}_{k,\beta,T}( u_{N+1}-u_N),
	\end{align*}
	for all $k\ge1$, $\beta,N>0$, $s\in(0\,,T]$, and $y\in\R$.
	On the other hand, 
	\begin{align*}
		&\left\| \left[ b_{N+1}(s\,,u_{N+1}(s\,,y))
			- b_N(s\,,u_N(s\,,y))\right] \mathbf{1}_{\Omega\setminus 
			G_{N+1}(s\,,y)}\right\|_k\\
		&\le \left\| b_{N+1}(s\,,u_{N+1}(s\,,y))
			\mathbf{1}_{\Omega\setminus G_{N+1}(s\,,y)}\right\|_k
			+ \left\| b_N(s\,,u_N(s\,,y))\mathbf{1}_{\Omega\setminus 
			G_{N+1}(s\,,y)}\right\|_k\\
		&\le \left[\left\| b_{N+1}(s\,,u_{N+1}(s\,,y))\right\|_{2k}
			+ \left\| b_N(s\,,u_N(s\,,y))\right\|_{2k}\right]
			\left[ 1-\P\left( G_{N+1}(s\,,y)\right)\right]^{1/(2k)}.
	\end{align*}
	We have used the following variation of the Cauchy--Schwarz inequality in the last line: 
	$\| X\mathbf{1}_F\|_k \le  \|X\|_{2k}[ \P(F)]^{1/(2k)}$
	for all $X\in L^k(\Omega)$ and all events $F\subset\Omega$.
	Set $C=4(\|u_0\|_{L^\infty(\R)} + 1).$
	Thanks to Proposition \ref{pr:moments},
	\begin{align*}
		&\left\| b_{N+1}(s\,,u_{N+1}(s\,,y))\right\|_{2k}
			+ \left\| b_N(s\,,u_N(s\,,y))\right\|_{2k}\\
		&\hskip1in\le L_b\left[ 
			\|u_{N+1}(s\,,y)\|_{2k} + \|u_N(s\,,y)\|_{2k}\right]
			\le 2C L_b\e^{512 L_{\sigma}^4 k^2s},
	\end{align*}
	uniformly for all $N,s>0$, $y\in\R$, and
	$k\ge\max(1\,, \frac12 L_b^{1/2}L_{\sigma}^{-2})$.
	This yields the following inequality:
	\begin{align*}
		&\left\| \left[ b_{N+1}(s\,,u_{N+1}(s\,,y))
			- b_N(s\,,u_N(s\,,y))\right] \mathbf{1}_{\Omega\setminus 
			G_{N+1}(s\,,y)}\right\|_k\\
		&\hskip1in \le 2C L_b\e^{512 L_{\sigma}^4 k^2s}
			\left[ \P\left\{ |u_{N+1}(s\,,y)| \ge \e^N\right\}\right]^{1/(2k)},
	\end{align*}
	valid uniformly for all $N,s>0$, $y\in\R$, and $k\geq 1$.
	Therefore, Proposition \ref{pr:prob:tail}
	yields
	\begin{align*}
		&\left\| \left[ b_{N+1}(s\,,u_{N+1}(s\,,y))
			- b_N(s\,,u_N(s\,,y))\right] \mathbf{1}_{\Omega\setminus 
			G_{N+1}(s\,,y)}\right\|_k\\
		&\hskip2.5in \le 2 L_b C\e^{512 L_{\sigma}^4 k^2s}
			\exp\left( - \frac{N^{3/2}}{128 k L_{\sigma}^2\sqrt s}\right),
	\end{align*}
	valid uniformly for all $s\in(0\,,T]$, $y\in\R$, 
	$N\ge 4 \log C\vee 256 T\max(4L_{\sigma}^4, L_b)$, 
	and $k\geq 1$. Thus, we find using \eqref{N'} that
	\begin{align}\nonumber
		I_1&\le L_{N,b}\e^{\beta t} \mathcal{N}_{k,\beta,T}( u_{N+1}-u_N)\int_{(0,t)\times\R} 
			\e^{-\beta (t-s)}p_{t-s}(y-x)
			\,\d s\,\d y\\
		&\quad +  2L_bC\int_{(0,t)\times\R}\e^{512 L_{\sigma}^4 k^2s}
			\exp\left( - \frac{N^{3/2}}{128 k L_{\sigma}^2\sqrt s}\right)
			p_{t-s}(y-x)\,\d s\,\d y\label{I11}\\\nonumber
		&\le \frac{L_{N,b}\e^{\beta t}}{\beta}\, \mathcal{N}_{k,\beta,T}( u_{N+1}-u_N)
			+ \frac{L_b C}{128 L_{\sigma}^4 k^2}
			\,\exp\left( 512 L_{\sigma}^4 k^2t
			- \frac{N^{3/2}}{128 k L_{\sigma}^2\sqrt t}\right),
	\end{align}
	uniformly for all $t\in(0\,,T)$ and $x\in\R$,
	provided that $N,k\ge c$, for a complicated looking but otherwise 
	unimportant number
	\begin{equation}\label{c}
		c=c(\|u_0\|_{L^\infty(\R)}\,,L_{\sigma}\,,T\,,L_b)>1.
	\end{equation}
	For later purposes, we pause to mention that the number $c$, while
	fixed, can be chosen to be as large as we wish. Owing to Condition \eqref{cond:LL:1},
	we therefore select $c=c(\|u_0\|_{L^\infty(\R)}\,,L_{\sigma}\,,T\,,L_b)$ large
	enough to additionally ensure that 
	\begin{equation}\label{c:N0}
		c > \sup_{N\ge N_0}\frac{\sqrt{L_{N,b}}}{L_{N,\sigma}^2},
		\quad\text{where}\quad
		N_0=\inf\left\{N>1:\ L_{N,\sigma}\ge1\right\}.
	\end{equation}
	The number $N_0$ is well defined and finite thanks to \eqref{L:N:sigma}.
	
	Next we study the quantity $I_2$.
	Let us appeal to the Burkholder-Davis-Gundy  inequality (see \cite{minicourse}) to find that
	\begin{align*}
		I_2^2 &\le 4k
			\int_0^t\d s\int_{-\infty}^\infty\d y\
			[p_{t-s}(y-x)]^2\left\| \sigma_{N+1}(s\,,u_{N+1}(s\,,y)) - 
			\sigma_N(s\,,u_N(s\,,y)) \right\|_k^2.
	\end{align*}
	As in the above truncation, we may write
	\begin{align*}
		&\left\| \left[ \sigma_{N+1}(s\,,u_{N+1}(s\,,y))
			- \sigma_N(s\,,u_N(s\,,y))\right] \mathbf{1}_{G_{N+1}(s\,,y)}\right\|_k\\
		&\le L_{N,\sigma}\e^{\beta s} \mathcal{N}_{k,\beta,T}( u_{N+1}-u_N)
			\qquad\forall \beta,N>0,\ s\in(0\,,T],\ y\in\R.
	\end{align*}
	Moreover, 
	\begin{align*}
		&\left\| \left[ \sigma_{N+1}(s\,,u_{N+1}(s\,,y))
			- \sigma_N(s\,,u_N(s\,,y))\right] \mathbf{1}_{\Omega\setminus G_{N+1}(s\,,y)}\right\|_k\\
		&\hskip.1in\le \left[\left\| \sigma_{N+1}(s\,,u_{N+1}(s\,,y))\right\|_{2k}
			+ \left\| \sigma_N(s\,,u_N(s\,,y))\right\|_{2k}\right]
			\left[ 1-\P\left( G_{N+1}(s\,,y)\right)\right]^{1/(2k)}.
	\end{align*}
	
	Thanks to Proposition \ref{pr:moments},
	\begin{align*}
		&\left\| \sigma_{N+1}(s\,,u_{N+1}(s\,,y))\right\|_{2k}
			+ \left\| \sigma_N(s\,,u_N(s\,,y))\right\|_{2k}\\
		&\hskip1in\le L_{\sigma}\left[ 
			\|u_{N+1}(s\,,y)\|_{2k} + \|u_N(s\,,y)\|_{2k}\right]
			\le 2L_{\sigma} C\e^{512 L_{\sigma}^4 k^2s},
	\end{align*}
	uniformly for all $N,s>0$, $y\in\R$, and $k\geq 1$.
	Therefore, Proposition \ref{pr:prob:tail} yields
	\begin{align*}
		&\left\| \left[ \sigma_{N+1}(s\,,u_{N+1}(s\,,y))
			- \sigma_N(s\,,u_N(s\,,y))\right] \mathbf{1}_{\Omega\setminus 
			G_{N+1}(s\,,y)}\right\|_k\\
		&\hskip1.7in\le 2L_{\sigma} C\e^{512 L_{\sigma}^4 k^2s}
			\exp\left( - \frac{N^{3/2}}{128 k L_{\sigma}^2\sqrt s}\right),
	\end{align*}
	uniformly for all $s\in(0\,,T]$, $y\in\R$, 
	$k\geq 1$, and 
	\begin{equation}\label{c_T}
		N\ge c_T := 4 \log C\vee 256 T\max\left(4L_{\sigma}^4, L_b\right).
	\end{equation}
	Thus, we find  that
	\begin{align*}
		&I_2^2 \le 8k L_{N,\sigma}^2 \e^{2\beta t}
			\left[ \mathcal{N}_{k,\beta,T}(u_{N+1}-u_N)\right]^2
			\int_0^t\d s\int_{-\infty}^\infty\d y\
			\e^{-2\beta(t-s)}[p_{t-s}(y-x)]^2 \\
		&\quad +  32 k L_b^2C^2\int_{(0,t)\times\R}\e^{1024 L_{\sigma}^4 k^2s}
			\exp\left( - \frac{N^{3/2}}{128 k L_{\sigma}^2\sqrt s}\right)
			[p_{t-s}(y-x)]^2\,\d s\,\d y\\
		&\leq \frac{4 kL_{N,\sigma}^2 \e^{2\beta t}}{\sqrt{\beta}}
			\left[ \mathcal{N}_{k,\beta,T}(u_{N+1}-u_N)\right]^2
			+ 128 k L_{\sigma}^2C^2 \e^{1024 L_{\sigma}^4 k^2t}
			\exp\left( - \frac{N^{3/2}}{128 k L_{\sigma}^2\sqrt t}
			\right) \sqrt{t},
	\end{align*}
	for every choice of $\beta>0$, $t\in(0\,,T)$, and $x\in\R$,
	provided that $N\ge \max(c\,,c_T)$ and $k\ge c$ -- in case it helps we recall that $c>1$ 
	and $c_T>0$ were defined respectively in \eqref{c} and \eqref{c_T}. 
	Combine this last inequality with \eqref{I11} and \eqref{u-u=I1I2} in order to see that
	\begin{align*}
		&\| u_{N+1}(t\,,x) - u_N(t\,,x) \|_k
			\le \e^{\beta t}\left[ L_{N,b}\beta^{-1} + 
			2\sqrt{k} L_{N,\sigma}\beta^{-1/4}
			\right] \mathcal{N}_{k,\beta,T}( u_{N+1}-u_N)\\
		&\hskip1in + C\left[ \frac{L_b}{128 L_{\sigma}^4 k^2} +
			\sqrt{128 k}t^{1/4} L_{\sigma} \right]
			\exp\left( 512 L_{\sigma}^4 k^2t
			- \frac{N^{3/2}}{128 k L_{\sigma}^2\sqrt t}\right),
	\end{align*}
	as long as $N\ge\max(c\,,c_T)$ and $k\ge c$. 
	
	Next, we make particular choices of $k\ge c$ and $\beta>0$ as follows:
	\begin{equation}\label{beta}
		k = c
		\quad\text{and}\quad
		\beta = 16A_0^4 k^2 L_{N,\sigma}^4,
	\end{equation}
	where 
	\begin{equation}\label{A_0}
		A_0 = \max\left( \sqrt{8}L_\sigma^4\,,4\right).
	\end{equation}
	We pause to emphasize that $\beta$ depends on $(N\,,T)$.
	In this way we find that
	\begin{align*}
		&\sup_{x\in\R}\| u_{N+1}(t\,,x) - u_N(t\,,x) \|_c
			\le \e^{\beta t}\left[ \frac{L_{N,b}}{16A_0^4c^2 L_{N,\sigma}^4} + 
			\frac{1}{A_0}
			\right] \mathcal{N}_{c,\beta,T}( u_{N+1}-u_N)\\
		&\hskip1in + C\left[ \frac{L_b}{128 L_{\sigma}^4 c^2} +
			\sqrt{128 c}t^{1/4} L_{\sigma} \right]
			\exp\left( 512 L_{\sigma}^4 c^2t
			- \frac{N^{3/2}}{128 c L_{\sigma}^2\sqrt t}\right),
	\end{align*}
	uniformly for all $T>0$, $t\in(0\,,T]$, and $N\ge \max(N_0\,,c\,,c_T)$;
	see \eqref{c:N0}. Because $A_0\ge 4$ -- see \eqref{A_0} --
	an application of \eqref{c:N0} yields the following inequality,
	\begin{align*}
		&\sup_{x\in\R}\| u_{N+1}(t\,,x) - u_N(t\,,x) \|_c
			\le \e^{\beta t}\left[ \frac{1}{16A_0^4} + 
			\frac{1}{A_0}
			\right] \mathcal{N}_{c,\beta,T}( u_{N+1}-u_N)\\
		&\hskip1in + C\left[ \frac{L_b}{128 L_{\sigma}^4 c^2} +
			\sqrt{128 c}t^{1/4} L_{\sigma} \right]
			\exp\left( 512 L_{\sigma}^4 c^2t
			- \frac{N^{3/2}}{128 c L_{\sigma}^2\sqrt t}\right)\\
		&\le \e^{\beta t}\left[ \frac{1}{4096} + 
			\frac{1}{4}
			\right] \mathcal{N}_{c,\beta,T}( u_{N+1}-u_N)\\
		&\hskip1in + C\left[ \frac{L_b}{128 L_{\sigma}^4 c^2} +
			\sqrt{128 c}t^{1/4} L_{\sigma} \right]
			\exp\left( 512 L_{\sigma}^4 c^2t
			- \frac{N^{3/2}}{128 c L_{\sigma}^2\sqrt t}\right),
	\end{align*}
	valid uniformly for all $T>0$, $t\in(0\,,T]$, and $N\ge \max(N_0\,,c\,,c_T)$.  
	Since $\frac{1}{4096} + 
	\frac{1}{4}<\frac12$, we may divide both sides of the preceding
	by $\exp(\beta t)$ and optimize over $t\in(0\,,T]$ in order to find that
	\begin{align}
		&\mathcal{N}_{c,\beta,T}( u_{N+1}-u_N)\label{N(u)<}\\\nonumber
		&\le 2C\left[ \frac{L_b}{128 L_{\sigma}^4 c^2} +
			\sqrt{128 c}T^{1/4} L_{\sigma} \right]
			\sup_{t\in(0,T]}
			\exp\left( -\left( \beta - 512 L_{\sigma}^4c^2 \right)t
			- \frac{N^{3/2}}{128 c L_{\sigma}^2\sqrt t}\right),
	\end{align}
	uniformly for all $T>0$, and $N\ge \max(N_0\,,c\,,c_T)$.  Thanks to 
	\eqref{c:N0}, \eqref{beta},
	and \eqref{A_0},
	\[
		\beta - 512 L_{\sigma}^4c^2 =
		c^2\left( 16A_0^4 L_{N,\sigma}^4 - 512 L_\sigma^4\right)
		\ge c^2\left( 16A_0^4  - 512 L_\sigma^4\right)>0,
	\]
	uniformly for all $N\ge N_0$.
	Now, let us consider the following function that appears in the exponent
	on the right-hand side of \eqref{N(u)<}:
	\[
		\psi(t) = \left( \beta - 512 L_{\sigma}^4 c^2 \right)t
		+ \frac{N^{3/2}}{128 k L_{\sigma}^2\sqrt t}
		\qquad(t>0).
	\]
	Because 
	\[
		\psi'(t)\le \beta - \frac{(N/T)^{3/2}}{256 c L_\sigma^2}
		= 16A_0^4 L_{N,\sigma}^4 c^2
		- \frac{(N/T)^{3/2}}{256 c L_\sigma^2}
		\quad\forall 0<t\le T,
	\]
	it follows from this that $\psi'<0$ everywhere on $(0\,,T]$ provided that
	\begin{equation}\label{N:Large}
		\frac{N}{L_{N,\sigma}^{8/3}} > 4096^{2/3} A_0^{8/3} c^{2}
		L_\sigma^{4/3} T =: N_T;
	\end{equation}
	the left-hand side being well defined for example when $N_T>N_0$.
	Condition \eqref{cond:LL:1} ensures that the left-hand side tends to
	infinity as $N\to\infty$. Therefore, \eqref{N:Large} holds for every $N>\max(N_0\,,N_T)$.
	It follows that
	\[
		\inf_{t\in(0,T]}\psi(t)=\psi(T)\quad\forall N>\max(N_0\,,N_T),\ T>0,
	\]
	whence it follows from \eqref{N(u)<} that,
	for every $T_0>0$ and $N\ge \max(N_0\,,N_{T_0}\,,c_0\,,c_{T_0})$,
	\begin{align*}
		&\mathcal{N}_{c,\beta,T}( u_{N+1}-u_N)\\
		&\le 2C\left[ \frac{L_b}{128 L_{\sigma}^4 c^2} +
			\sqrt{128 c}T^{1/4} L_{\sigma} \right]
			\exp\left( -\left( \beta - 512 L_{\sigma}^4 c^2 \right)T
			- \frac{N^{3/2}}{128 c L_{\sigma}^2\sqrt T}\right),
	\end{align*}
	uniformly for every $T\in(0\,,T_0)$. In order to ensure the uniformity statement of
	$T$, we have also used the fact that $T\mapsto c_T$ is increasing; see
	\eqref{c_T}.
	We now apply Lemma \ref{lem:sum} in order to deduce from the above
	and \eqref{N'} that, for every $T>0$ fixed,
	\begin{equation}\label{case1}
		\limsup_{N\to\infty} N^{-3/2}
		\log\adjustlimits\sup_{t\in(0,T]}\sup_{x\in\R}
		\|u_{N+1}(t\,,x)-u_N(t)\|_k
		<0
		\qquad\forall k\ge1.
	\end{equation}
	When $k\in[1\,,c]$ this follows from the preceding.
	For general $k$ it follows from a  relabeling $[k\leftrightarrow c$],
	and an appeal to the fact that $c$ can be as large as we wish -- see the comments that follow
	\eqref{c}. This proves our original goal \eqref{goal:1}.
	The remainder of the argument is technically simpler.
	
	The preceding proves that
	\begin{equation}\label{u}
		u(t\,,x) = \lim_{N\to\infty} u_N(t\,,x)
		\quad\text{exists in $L^k(\Omega)$}
		\qquad\forall k\ge 1,
	\end{equation}
	and the rate of convergence does not depend on $t\in(0\,,T)$ nor on $x\in\R$.
	As a direct consequence $u$ is $L^k(\Omega)$-continuous. This ensures that $u$ has a predictable version.
	Therefore it remains to prove that, for every $t\in(0\,,T)$ and $x\in\R$,
	\begin{equation}
		\lim_{N\to\infty} \I^N_{b_N}(t\,,x) = \I_b(t\,,x)
		\quad\text{and}\quad
		\lim_{N\to\infty}\I^N_{\sigma_N}(t\,,x) = \I_\sigma(t\,,x),\label{double}
	\end{equation}
	where both of the limits hold in $L^2(\Omega)$,
	and the random fields $\I_b$, $\I_\sigma$, $\I^N_b$, and $\I^N_\sigma$
	were \textcolor{black}{defined in 
	\eqref{I_b} and \eqref{I_b^N}}. Thanks to \eqref{u}
	and \eqref{mild}, this proves that $u$ is a mild solution to \eqref{SHE}.
	
	Observe that \begin{equation*} \begin{split} |b_N(t\,,z)-b(t\,,z)| &\le
	\mathbf{1}_{\R\setminus[-\exp(N),\exp(N)]}(z) ( |b(t\,,z)| + |b_N(t\,,z)| )\\
	&\le 2L_b(1+|z|)\mathbf{1}_{\R\setminus[-\exp(N),\exp(N)]}(z),
	\end{split}
	\end{equation*}
	 for all $z\in\R$
	and $t,N>0$.
	This yields
	\begin{align*}
		&
			\left\| \I^N_{b_N}(t\,,x) - \int_{(0,t)\times\R} p_{t-s}(y-x)b(s\,,u_N(s\,,y))
			\,\d s\,\d y \right\|_2\\
		&
			\le\int_{(0,t)\times\R} p_{t-s}(y-x)\left\| b_N(s\,,u_N(s\,,y)) - b(s\,,u_N(s\,,y))\right\|_2
			\,\d s\,\d y\\
		&
			\le2L_b\int_{(0,t)\times\R} p_{t-s}(y-x)\left( \sqrt{\E\left( 1+|u_N(s\,,y)|^2;
			|u_N(s\,,y)|>\e^N\right)}\right)
			\,\d s\,\d y.
	\end{align*}
	If $X\ge0$ is a random variable and $A>0$ is a constant, then
	we apply the Cauchy-Schwarz inequality and Chebyshev's inequality
	back to back in order to find that
	$$\E(X^2;X>A)\le \sqrt{\E(X^4)\P\{|X|>A\}}
	\le A^{-2}\E(X^4).$$
	Therefore, Propositions \ref{pr:moments} and \ref{pr:moments:bdd} assure us that
	\[
		\lim_{N\to\infty}\sup_{s\in(0,T)}\sup_{y\in\R}
		\E\left( 1+|u_N(s\,,y)|^2;
		|u_N(s\,,y)|>\e^N\right)=0,
	\]
	and hence
	$\I^N_{b_N}(t\,,x) - \int_{(0,t)\times\R} p_{t-s}(y-x)b(s\,,u_N(s\,,y))
	\,\d s\,\d y\to 0$ as $N\to\infty$,
	where the convergence takes place in $L^2(\Omega)$. Therefore, the first assertion of
	\eqref{double} would follow once we can show that
	\begin{equation}\label{double1:goal}
		\lim_{N\to\infty}\int_{(0,t)\times\R} p_{t-s}(y-x)b(s\,,u_N(s\,,y))
		\,\d s\,\d y =\I_b(t\,,x),
	\end{equation}
	where the convergence takes place in $L^2(\Omega)$. 
	Since $b$ has at-most linear growth and
	$\|b(s\,,u_N(s\,,y)) - b(s\,,u(s\,,y))\|_2 \le
	\|b(s\,,u_N(s\,,y))\|_2 + \| b(s\,,u(s\,,y))\|_2,$
	Propositions \ref{pr:moments} and \ref{pr:moments:bdd}
	ensure that  $\|b(s\,,u_N(s\,,y)) - b(s\,,u(s\,,y))\|_2$ is bounded uniformly in $s\in(0\,,T)$,
	$N>0$, and $y\in\R$. Thanks to \eqref{u}, uniform integrability, and
	the continuity of $b$, 
	\[
		\lim_{N\to\infty}b(s\,,u_N(s\,,y))=b(s\,,u(s\,,y))
		\quad\text{in $L^2(\Omega)$, for every $s>0$ and $y\in\R^d$. }
	\] 
	Therefore, the dominated convergence theorem yields
	\[
		\lim_{N\to\infty}\int_{(0,t)\times\R} p_{t-s}(y-x) \left\| b(s\,,u_N(s\,,y)) - b(s\,,u(s\,,y))\right\|_2
		\,\d s\,\d y=0,
	\]
	for all $t\in(0\,,T)$ and $x\in\R$.
	The triangle inequality now yields \eqref{double1:goal} and hence the first assertion 
	of \eqref{double}. 
		Similarly, $$|\sigma_N(t\,,z)-\sigma(t\,,z)| 
	\le 2L_{\sigma}(1+|z|)\mathbf{1}_{\R\setminus[-\exp(N),\exp(N)]}(z)$$ for all $z\in\R$
	and $t,N>0$.
	Then, by the $L^2(\Omega)$-isometry
	of stochastic integrals,
	\begin{align*}
		&\left\| \I^N_{\sigma_N}(t\,,x) - \int_{(0,t)\times\R} p_{t-s}(y-x)\sigma(s\,,u_N(s\,,y))
			\,W(\d s\,\d y) \right\|_2^2\\
		&\le\int_{(0,t)\times\R} [p_{t-s}(y-x)]^2 \left\| \sigma_N(s\,,u_N(s\,,y)) 
			- \sigma(s\,,u_N(s\,,y))\right\|_2^2
			\,\d s\,\d y\\
		&\le4L_{\sigma}^2\int_{(0,t)\times\R} [p_{t-s}(y-x)]^2 \E\left( 1+|u_N(s\,,y)|^2;
			|u_N(s\,,y)|>\e^N\right)
			\,\d s\,\d y.
	\end{align*}
	As before, Propositions \ref{pr:moments} and \ref{pr:moments:bdd} assure us that
	\[
		\lim_{N\to\infty}\sup_{s\in(0,T)}\sup_{y\in\R}
		\E\left( 1+|u_N(s\,,y)|^2;
		|u_N(s\,,y)|>\e^N\right)=0,
	\]
	and hence $\I^N_{\sigma_N}(t\,,x) - \int_{(0,t)\times\R} p_{t-s}(y-x)\sigma(s\,,u_N(s\,,y))
	\,W(\d s\,\d y)\to 0$ in $L^2(\Omega)$ as $N\to\infty$. Therefore, we are left to show that
	\[
		\lim_{N\to\infty}\int_{(0,t)\times\R} p_{t-s}(y-x)\sigma(s\,,u_N(s\,,y))
		\,W(\d s\,\d y) = \I_\sigma(t\,,x),
	\]
	where the convergence takes place in $L^2(\Omega)$. 
	Since $\sigma$ has at-most linear growth and 
	\[
		\|\sigma(s\,,u_N(s\,,y)) - \sigma(s\,,u(s\,,y))\|_2 \le
		\|\sigma(s\,,u_N(s\,,y))\|_2 + \| \sigma(s\,,u(s\,,y))\|_2,
	\]
	Propositions \ref{pr:moments} and \ref{pr:moments:bdd}
	ensure that  $\|\sigma(s\,,u_N(s\,,y)) - \sigma(s\,,u(s\,,y))\|_2$ is bounded uniformly in $s\in(0\,,T)$,
	$N>0$, and $y\in\R$. Thanks to \eqref{u}, uniform integrability, and
	the continuity of $\sigma$, 
		$\lim_{N\to\infty}\sigma(s\,,u_N(s\,,y))=\sigma(s\,,u(s\,,y))$
		in $L^2(\Omega)$, for every $s>0$ and $y\in\R^d$.  
	Therefore, 
	\begin{align*}
		&\left\| \int_{(0,t)\times\R} p_{t-s}(y-x)\sigma(s\,,u_N(s\,,y))
			\,W(\d s\,\d y) - \I_\sigma(t\,,x) \right\|_2^2\\
		&\le\int_{(0,t)\times\R} [p_{t-s}(y-x)]^2 \left\| \sigma(s\,,u_N(s\,,y)) -
			\sigma(s\,,u(s\,,y))\right\|_2^2
			\,\d s\,\d y,
	\end{align*}
	which converges to 0 as $N\rightarrow \infty$ by dominated convergence, 
	for all $t\in(0\,,T]$ and $x\in\R$.
	The triangle inequality now yields \eqref{double1:goal}. This
	verifies the second assertion of \eqref{double} and completes 
	the proof of Theorem \ref{th:exists}
	in the case that $L_\sigma>0$. It remains to study the same problem when
	$\sigma$ is bounded.

	The proof in the case that $\sigma$ is bounded is structurally
	similar to the derivation of the first part, except
	we replace the tail bound \eqref{eq:tail:1} with \eqref{eq:tail:2}
	in order to obtain the following variant of \eqref{case1}:
	\[
		\limsup_{N\to\infty} \e^{-2N}
		\log\adjustlimits\sup_{t\in(0,T]}\sup_{x\in\R}
		\|u_{N+1}(t\,,x)-u_N(t)\|_k
		<0
		\qquad\forall k\ge1.
	\]
	The remainder of our
	derivation of the $\sigma$-bounded case is the same
	as in the first portion of the proof [$L_\sigma>0$] and therefore omitted. This completes
	the proof.
\end{proof}

\begin{proof}[Proof of uniqueness]
	We only give an outline of the 
	proof in the case that $L_\sigma>0$, since the proof of uniqueness
	is essentially a simplied modification of the proof of existence, and because the other case where
	$\sigma$ is bounded is similar.
	
	Let $u,v$ two solutions to (\ref{SHE}) with same initial
	condition $u_0$ satisfying Assumption \ref{cond-initial}. Recall that
	$u$ and $v$ satisfy the mild formulation (\ref{mild_SHE})
	with $\sigma$ and $b$ satisfying Assumption \ref{cond-dif}. 
	Then, using Burkholder-Davis-Gundy inequality (see \cite{minicourse}), 
	we find that
	for all $t>0$, $x \in \R$, and $k \geq 1$,
	$$\Vert u(t\,,x)-v(t\,,x)\Vert_k \leq I_1+I_2+I_3+I_4,$$
	where
	\begin{align*}
		I_1&=\int_{(0,t)\times\R} p_{t-s}(y-x)\left\| (b(s\,,u(s\,,y))
			- b(s\,,v(s\,,y))){\bf 1}_{A_N(s, y)}\right\|_k \d s\,\d y,\\
		I_2^2&=4k \int_0^t\d s\int_{-\infty}^\infty\d y\
			[p_{t-s}(y-x)]^2\left\| (\sigma(s\,,u(s\,,y))
			- \sigma(s\,,v(s\,,y))){\bf 1}_{A_N(s, y)}\right\|_k^2,\\
		I_3&=\int_{(0,t)\times\R} p_{t-s}(y-x)\left\| (b(s\,,u(s\,,y))
			- b(s\,,v(s\,,y))){\bf 1}_{\Omega \setminus A_N(s, y)}\right\|_k\d s\,\d y, \\
		I_4^2&=8k \int_0^t\d s\int_{-\infty}^\infty\d y\
			[p_{t-s}(y-x)]^2\left\| (\sigma(s\,,u(s\,,y))
			- \sigma(s\,,v(s\,,y))){\bf 1}_{\Omega \setminus A_N(s, y)}\right\|_k^2,
	\end{align*}
	where for $N,s>0$ and $y \in \R$,
	$$
		A_N(s\,, y)=\left\{ \omega\in\Omega:\, |u(s\,,y)|(\omega) \le
		\e^N,\ |v(s\,,y)|(\omega) \le \e^N\right\}.
	$$
	
	Observe that the moment bounds obtained in Propositions
	\ref{pr:moments} and \ref{pr:moments:bdd} only use the linear
	growth constants $L_{\sigma}$ and $L_b$. Therefore, they
	also hold when $u_N$ is replaced by $u$ or $v$.
	The same is true for the tail estimates obtained Proposition
	\ref{pr:prob:tail}, which also hold replacing $u_{N+1}$ by $u$ or $v$.
	Therefore, we can proceed as in the proof of existence 
	but appeal to the local Lipschitz condition on $b$ and $\sigma$; that is,
	we recall \eqref{L_N} and write
	$$
		\left\| \left[ b(s\,,u(s\,,y))
		- b(s\,,v(s\,,y))\right] \mathbf{1}_{A_N(s,y)}\right\|_k
		\le L_{N,b}\left\| u(s\,,y) - v(s\,,y) \right\|_k,
	$$
	and  proceed similarly for $\sigma$. Because 
	$\P(\Omega\setminus A_N(s\,, y))$
	is at most $\P \{|u(t\,, y)|\geq \e^N \}+\P \{|v(t\,, y)|\geq \e^N \},$
	similar computations to those in the proof of existence imply that for all $T>0$,
	\begin{align*}
		\| u(t\,,x) - v(t\,,x) \|_k
			&\le \e^{\beta t}\left[ L_{N,b} \beta^{-1} + 
			2\sqrt{k} L_{N,\sigma} \beta^{-1/4}
			\right] \mathcal{N}_{k,\beta,T}( u-v)\\
		&+ C\left[ \frac{L_b}{128 L_{\sigma}^4 k^2}
			+\sqrt{128 k} T^{1/4} L_{\sigma} 
			\right] \exp\left( 512 L_{\sigma}^4 k^2t
			- \frac{N^{3/2}}{128 k L_{\sigma}^2\sqrt t}\right),
	\end{align*}
	uniformly for all $\beta>0$, $t \in (0,T)$,  $x \in \R$, and $N, k \geq c$, 
	where $C=4(\|u_0\|_{L^\infty(\R)} +1)$
	and $c=c(\|u_0\|_{L^\infty(\R)}\,,L_{\sigma}\,,T\,,L_b)>0$. Now we fix
	the parameters $\beta$ and $k$ as in \eqref{beta} and analyze the preceding
	exactly as was done in the proof of
	existence in order to conclude that
	\[
		\sum_{N=1}^\infty\adjustlimits\sup_{t\in(0,T]}\sup_{x\in\R}
		\|u(t\,,x)-v(t\,,x)\|_2<\infty\qquad\forall T>0.
	\]
	Since the summand does not depend on $N$, it must be zero.
	This concludes the proof.
\end{proof}

\subsection*{Acknowledgements}
We thanks Professors Raluca Balan and Tusheng Zhang for their encouraging remarks
and for pointing out a serious error in an earlier draft. The correction of that error
led to Assumption \ref{cond:lip}.

\bibliography{Foon-Nual}

\end{document}